\documentclass[a4paper]{amsart}

\usepackage[fontsize=12pt]{scrextend}
\usepackage{blindtext}

\usepackage[english]{babel}

\usepackage{amsmath}
\usepackage{amsthm}
\usepackage{amssymb}
\usepackage{amscd}
\usepackage[all]{xy}
\usepackage{euscript}

\theoremstyle{plain}
\newtheorem{theorem}{Theorem}[section]
\newtheorem*{theoremA}{Theorem}
\newtheorem{lemma}[theorem]{Lemma}
\newtheorem{proposition}[theorem]{Proposition}

\theoremstyle{definition}

\newtheorem{ex-constr}[theorem]{Construction}

\renewcommand{\dim}{\mathrm{dim}\,}

\renewcommand{\P}{{\mathbb P}}

\newcommand{\CH}{\mathrm{CH}}

\title[The variety of triangles]{
On the variety of triangles for a hyper-K\"ahler fourfold constructed by Debarre and Voisin}
\author{Ivan Bazhov}
\address{Institut de Mathématiques de Jussieu, 4 Place Jussieu, 75005, France}
\email{ibazhov@gmail.com}
\keywords{Hyper-K\"ahler varieties; Lagrangian subvarieties; Fano variety of lines; Algebraic cycles; Chow rings}
\date{}

\begin{document}

\begin{abstract}
We study the similarities between the Fano varieties of lines on a cubic fourfold, a hyper-K\"ahler fourfold studied by Beauville and Donagi, and the hyper-K\"ahler fourfold constructed by Debarre and Voisin in \cite{DV}.
We exhibit an analog of the notion of "triangle" for these varieties and prove that the 6-dimensional variety of "triangles" is a Lagrangian subvariety in the cube of the constructed hyper-K\"ahler fourfold.
\end{abstract}

\maketitle{}

\section*{Introduction}
By definition a compact K\"ahler manifold $X$ is a hyper-K\"ahler manifold 
if $X$ is simply connected and $H^0(X,\Omega_X^2)$ is of dimension $1$, generated by a holomorphic 2-form $\sigma$, which is non-degenerated at any point of $X$. The 2-form $\sigma$ is called the symplectic holomorphic form of $X$. It is defined up to a multiplicative constant.

Beauville in \cite{B84} provides two series of families of examples, for each even complex dimension: (a) the $n-$punctual Hilbert scheme $S^{[n]}$ of a K3 surface $S$ and (b) the fiber at the origin of the Albanese map of the $(n+1)-$st punctual Hilbert scheme of an abelian surface. All of the irreducible hyper-K\"ahler manifolds constructed later on are deformation-equivalent to one of Beauville's examples, with two exceptions: O'Grady examples in dimension 6 and in dimension 10 (see \cite{OG1,OG2}).

We note that the varieties in Beauville's examples have Picard number two, while a general algebraic deformation of a hyper-K\"ahler manifold has Picard number one.  There are not so many available explicit constructions of these general deformations with Picard number one. Only four such families, each of which is 20-dimensional and parametrizes general polarized deformations of the second punctual Hilbert scheme of a K3 surface, are known:

\begin{enumerate}
\item (Beauville and Donagi, \cite{BD}) The Fano variety of lines of a cubic fourfold. It was proven in \cite{BD} that the variety $F(X)$ of lines on a smooth cubic hypersurface $F\subset \P^5$ is an algebraic hyper-K\"ahler fourfold. It gives a 20-dimensional moduli space of fourfolds.
\item (Iliev and Ranestad, \cite{IR1,IR2}) The variety $V(X)$ of sum of powers of a general cubic $X\subset \P^5$. It was proven in \cite{IR1,IR2} that it is another algebraic hyper-K\"ahler fourfold, with 20 moduli.
\item (O'Grady, \cite{OG4}) O'Grady constructed a 20-parameter family of hyper-K\"ahler algebraic fourfolds. They are quasi-\'etale double covers of certain singular sextic hypersurfaces constructed by Eisenbud, Popescu, and Walter. 
\item (Debarre and Voisin, \cite{DV}) Using Grassmann geometry another 20-dimensional family of hyper-K\"ahler varieties which are deformations of $S^{[2]}$ for $S$ of genus 12 is constructed.  
\end{enumerate}


We study a hyper-K\"ahler four-dimensional manifold $F$ constructed by Debarre and Voisin from a hyperplane section $X$ in $Gr(3,V_{10})$. The construction is very similar to the construction of Fano variety of lines for a cubic fourfold (see the next section). Like to the case of Fano variety of lines, where a "triangle" is three lines in the cubic fourfold 
having non-trivial pairwise intersections, we introduce a notion of triangle on $X$ and define a corresponding subvariety $I_3\subset F\times F\times F$.  Our main result is the following theorem. 

\begin{theoremA}
(Theorem \ref{th_42} (\ref{th_42_2}))
The $6-$dimensional subvariety $I_3$ is Lagrangian for 2-form $\sum_{i=1}^{3}\pi^*_i\sigma_F$, where $\pi_i:F^3\to F$ are the natural projections and $\sigma_F$ is a holomorphic 2-form on $F$.
\end{theoremA}

Our proof uses the fact that the cycle $[W_1]+[W_2]+[W_3]$ on $X$ corresponding to a point of $I_3$ is constant in $\CH_9(X)$ (i.e., does not depend on a choice of point in $I_3$). The similar result is obviously true for a cubic fourfold: any triangle is just a restriction of some plane to the cubic hypersurface.

Lagrangian subvarieties are related to constant cycle subvarieties. The notion of constant cycles subvarieties was introduced in \cite{H14} and used in \cite{Lin15, Voisin15} to study $\CH_0(X)$ for hyper-K\"ahler manifolds. A constant cycle subvariety $Y$ of $X$ is a subvariety of $X$ such that any two points of $Y$ represent the same class in $\CH_0(X)$. It was shown that any constant cycle subvariety of a hyper-K\"ahler variety is Lagrangian. However, we do not expect to prove that $I_3$ is a constant cycle subvariety, because there is no similar result for a cubic fourfold (see \cite[Theorem 20.5]{SV} for details).

We hope that the presented result will allow to attack the Beauville-Voisin conjecture for the hyper-K\"ahler manifolds constructed by Debarre and Voisin, which is already proved for Fano varieties of lines of cubic fourfold \cite{Voisin08}. 

{\bf Acknowledgement.} I am grateful to Claire Voisin for her kind help and guidance during the work. I am also grateful to
Samuel Boissi\`ere, Kieran O'Grady, and Laurent Manivel for reading the proof as a part of my thesis and for useful comments and remarks.

\section{Construction and the statements of the main results}

\subsection{Construction}

Let us recall the construction of \cite{DV}. Let $G(3,V_{10})$ be the Grassmann variety of 3-dimensional vector subspaces in a 10-dimensional vector space $V_{10}$ and let $X$ be a hyperplane section in $Gr(3,V_{10})$. The variety $X$ is defined by a 3-form
$$
\alpha_X=\sum \alpha_{ijk} e_i^*\wedge e_j^*\wedge e_k^*\in \Lambda^3 V_{10}^{*}
$$
where $(e^*_i)$ is a basis of the dual vector space $V^*_{10}$. 

The variety $F(X)$ is then defined as the subvariety of $Gr(6,V_{10})$ of all 6-dimensional spaces $V_6\subset V_{10}$ such that the form $i_{V_6}^*\alpha_X\in\Lambda^3 V_{6}^{*}$ is zero, where $i_{V_6}^*:\Lambda^3V_{10}^*\to \Lambda^3V_{6}^*$ is that natural map. Equivalently, for any 3-dimensional $V_3\subset V_6$ the restriction $i_{V_3}^*\alpha_X$ is zero and hence $Gr(3,V_6)\subset X$. We thus have a natural universal diagram:
\begin{equation}
\label{diag_new_fano}
\xymatrix{
U \ar[d]^{p} \ar[r]^{q\phantom{MMM}} &X\subset Gr(3,V_{10})\\
F(X)\subset Gr(6,V_{10})&},
\end{equation}
where $U$ is the universal variety consisting of pairs $(V_3,V_6)$ such that $V_3\subset V_6$ and $i_{V_6}^*\alpha_X$ is zero. For $[V_6]\in F(X)$ we will denote $Z_{V_6}:= Gr(3,V_6)\subset X$, a nine-dimensional subvariety of $X$ whose class is $U_*[V_6]$.  

\begin{theorem}
[\cite{DV}]
For $\alpha_X$ general, 
the variety $F(X)$ is an irreducible hyper-K\"ahler manifold of dimension four. More precisely, endowed with the Pl\"ucker line bundle, it is deformation-equivalent to the second punctual Hilbert scheme $S^{[2]}$ of a K3 surface $S$ of genus 12, endowed with the line bundle whose pull-back to $\widetilde{S\times S}$ is $(\mathcal O_S (1)\boxtimes\mathcal O_S (1))^{10}(-33\widetilde E)$. 
\end{theorem}
In this theorem $\widetilde{S\times S} \to S\times S$ is the blow-up of the diagonal, $\widetilde {E}$ is the exceptional divisor, and the pull-back is via the canonical double cover $\widetilde{S\times S} \to S^{[2]}$.

The goal of this paper is to study the variety $F(X)$ and its similarities with the variety of lines of a cubic fourfold, which has the following similar construction. Let $Y\subset \P^5=Gr(2,6)$ be a smooth hypersurface of degree $3$, and $F(Y)\subset Gr(2,6)$ be the variety of lines contained in $Y$:
$$
\xymatrix{
U \ar[d]^p \ar[r]^q &Y\subset \P^5\\
F(Y)\subset Gr(2,6)&},
$$
here $U$ is the universal variety consisting of pairs $(x,[l])$, where $x\in X$, the line $l$ is contained in $Y$, and $x\in l$. 

The notation $F(X)$ for the fourfold constructed by Debarre and Voisin and $F(Y)$ for the Fano variety of lines may look confusing, but the author decided to use them in order to emphasise the similarity between two varieties. The Fano variety of lines does not appear below, so $F(X)$ will unequivocally refer to the fourfold constructed by Debarre and Voisin.

\subsection{Statements of main results}

In this subsection we announce Theorems \ref{th_42} and \ref{th_43}, which provide evidences of similarities between the Fano varieties of lines on a cubic fourfold and the hyper-K\"ahler fourfold constructed by Debarre and Voisin. The rest of the paper will be devoted to the proofs of these theorems.

Following the ideas used in \cite{Voisin08} for the Fano variety of lines in a cubic fourfold, we are going to consider the incidence variety $I$ of pairs $([W_1],[W_2])\in F(X)\times F(X)$ such that the corresponding subvarieties $Z_{W_1}$ and $Z_{W_2}$ on $X$ have a common point. This common point, represented be a $3-$dimensional space, can be viewed as a point on the diagonal $\Delta_X$ in $X\times X$, and since we expect that $(p,p)(q,q)^{-1}\Delta_X$ is reducible and contains a diagonal as a component, we then define $I$ in the following way:
$$
I\subset(p,p)(q,q)^{-1}\Delta_X,
$$
$$
I=\overline{\left((p,p)(q,q)^{-1}\Delta_X\right)\smallsetminus \Delta_{F(X)}},
$$
where $p$ and $q$ were defined by the diagram (\ref{diag_new_fano}). The variety $I$ has a stratification:
$$
I=I^{(3)}\cup I^{(4)}\cup I^{(5)}\cup I^{(6)},
$$
where $I^{(i)}$ is the subvarity of $I$ consisting of pairs $(W_1,W_2)$ with $i$-dimensional intersection. For general $\alpha_X$, calculations give 
$$
\dim I^{(3)}=\dim I=6,\ \dim I^{(4)}=4,\ \dim I^{(5)}=3,
$$
and away from $\Delta_{F(X)}$ which could be contained in $I$, we have $\dim I^{(6)}\leq 3$. 

We define the variety of "triangles" as the closure of $I_3^o$:
\begin{multline*}
I^o_3=\{([W_1],[W_2],[W_3])| \dim(W_i\cap W_j)\geq3\ \forall i,j\\
\mbox{and}\ \dim\left( W_1\cap W_2\cap W_3\right)=0\}\subset F(X)\times F(X)\times F(X),
\end{multline*}
$$
I_3=\overline {I^o_{3}}\subset F(X)\times F(X)\times F(X).
$$
In Lemma \ref{l_45} below we will show that the natural projection $\pi_{12}: I_3\to I$ has degree one. One can also consider a bigger variety defined as
$$
I'_3=\{([W_1],[W_2],[W_3])| \dim(W_i\cap W_j)\geq3\ \forall i,j\}\subset F(X)\times F(X)\times F(X).
$$
and show that $I_3$ is an irreducible component of $I_3'$.

\begin{theorem}
\label{th_42}
\begin{enumerate}
\item \label{th_42_1} There exists a cycle $\gamma\in \CH_{10}(Gr(3,10))$ such that for any $([W_1],[W_2],[W_3])\in I_3$, the sum $Z_{W_1}+Z_{W_2}+Z_{W_3}\in \CH_9(X)$ is the restriction to $X$ of $\gamma$. 
\item \label{th_42_2}The $6-$dimensional subvariety $I_3\subset F(X)\times F(X)\times F(X)$ is a Lagrangian subvariety for the $(2,0)-$form $pr_1^*\sigma_{F(X)}+pr_2^*\sigma_{F(X)}+pr_3^*\sigma_{F(X)}$,
where $\sigma_{F(X)}$ denotes a $(2,0)-$form on $F(X)$ generating $H^{2,0}(F(X))$.
\end{enumerate}
\end{theorem}

The proof of Theorem \ref{th_42} starts in the next section and goes until the end of the paper.

Another similarity between Fano varieties of lines and Debarre-Voisin fourfolds is given in the next theorem (see \cite[Proposition 3.3]{Voisin08} for the corresponding results for the Fano variety of lines).

\begin{theorem}
\label{th_43}
\label{qu_relation}
There is a quadratic relation in $\CH^4(F(X)\times F(X))$ of the form
$$
I^2=\alpha \Delta_{F(X)}+\beta I^{(4)}+I\cdot \Gamma_1+ \Gamma_2+ q^*p_*\Gamma_3,
$$
where $\Gamma_1\in \CH^2(Gr(6,10)^2)|_{F(X)\times F(X)}$, $\Gamma_2\in \CH^4(Gr(6,10)^2)|_{F(X)\times F(X)}$ and $\Gamma_3$ is proportional to $\Delta_{X*}c_2(Gr(3,10)) \in \CH(X\times X)$.
\end{theorem}

We will prove Theorem \ref{th_43} in the next section.

\section{Proofs of main results}
\label{subsection_quadratic}

We start with the proof of Theorem \ref{th_42}.

\begin{proof}[Proof of Theorem \ref{th_42}.]
Item (\ref{th_42_1}) will be proved in Proposition \ref{triangle} below.

From Lemma \ref{lemma_I_6} and Lemma \ref{l_45} below it is follows that $I_3$ is a $6-$di\-men\-sional subvariety. Let us show that (\ref{th_42_1}) implies (\ref{th_42_2}). 

Let $\tilde I_3$ be a desingularisation of $I_3$. Let $T$ be the natural correspondence between $\tilde I_3$ and $F(X)$. In particular, we have a map $T_*:\CH_0(\tilde I_3)\to\CH_0(F(X))$. Note that, by (\ref{th_42_1}), the image of the composition 
$$
U_*\circ T_*: \CH_0(\tilde I_3)\to \CH_0(F(X))\to \CH_9(X)
$$
is $\mathbb Z$. Therefore, by the generalisation of Mumford's theorem \cite{M}, the map
$$
(U\circ T)^*:H^{11,9}(X)\to H^{2,0}(\tilde I_3)
$$
is zero. But $U^*H^{11,9}(X)=H^{2,0}(F(X))$ by \cite{DV},
so $\sigma_{F(X)}=U^*\eta_X$ for some $\eta_X\in H^{11,9}(X)$ and thus $\left(pr_1^*\sigma_{F(X)}+pr_2^*\sigma_{F(X)}+pr_3^*\sigma_{F(X)}\right)|_{\tilde I_3}=(U\circ T)^*\eta_X=0$, which proves that the subvariety $I_3$ is Lagrangian.
\end{proof}

\begin{proof}[Proof of Theorem \ref{qu_relation}]
We follow the line of the proof of the similar statement for the Fano variety of lines (see \cite[Proposition 3.3]{Voisin08}).

We are going to establish a relation of the following form:
$$
I_o^2=I_o\cdot \Gamma_1+ \Gamma_2+\Gamma_3,
$$
in $\CH^4(F\times F\smallsetminus\left(\Delta_{F(X)}\cup I^{(4)}\right))$, where $I_o$ is the restriction of $I$ to $(F\times F\smallsetminus\left(\Delta_{F(X)}\cup I^{(4)}\right))$. The result will follow by the localisation exact sequence.

We recall that $I$ is the image in $F\times F$ of $\tilde I=(q,q)^{-1}\Delta_X$ under the projection $(p,p): \tilde I\to I$. Since $(p,p)$ is an isomorphism away from $\Delta_{F(X)}\cup I^{(4)}$, we have a local isomorphism between $I_o$ and $\tilde I_o=(p,p)^{-1}(I_o)$. 

We denote some Chern classes in the following short way:
$$
c_i^j(3)=pr_j^*c_i(Gr(3,10))\in\CH_i(Gr(3,10)),
$$
$$
c_i^j(6)=pr_j^*c_i(Gr(6,10))\in\CH_i(Gr(6,10)).
$$  
When we speak about $\CH_*(U\times U)$, where $U$ is the universal variety in the diagram (\ref{diag_new_fano}), to keep notation simple we will denote $(p,p)^*(c_i^j(6))$ as $c_i^j(6)$ and $(q,q)^*(c_i^j(3))$ as $c_i^j(3)$.

We have the normal sequence:
$$
0\to T_{U\times U/F\times F|_{\tilde I_0}}\to N_{\tilde I_0/U\times U}\to (p,p)^*N_{I_0/F\times F}\to 0,
$$
therefore $(p,p)^*N_{I_o/F\times F}$ can be expressed as a polynomial in the Chern classes of the normal bundle $N_{\tilde I_o/U\times U}$ and in Chern classes of $T_{U\times U/F\times F|_{\tilde I_o}}$. The later ones are polynomial in $c_i^j(6)$ and $c_i^j(3)$. Next, we see that $\tilde I=(q,q)^{-1}(\Delta_X)$, therefore
$$
c_i(N_{\tilde I_o/U\times U})=(q,q)^*c_i(T_X),
$$
but $c_i(T_X)$ are polynomial in $c_j(Gr(3,V_{10}))$. So we have that
$$
I_o^2=(p,p)_*(P\cdot \tilde I),
$$
where $P$ is a quadratic polynomial in $c_i^j(6)$ and $c_i^j(3)$. The polynomial $P$ can be divided in three parts:
\begin{enumerate}
\item the part containing only $c_i^j(6)$. Since all these terms have from $(p,p)^*(c)$ for some $c\in\CH_2(F(X)\times F(X))$, the intersection with $\tilde I$ and projection $(p,p)_*$ gives the term $\Gamma_1\cdot I_o$.
\item the part divisible by $c_1^j(3)$.  The term $c_1^j(3)$ has the from $(q,q)^*(c)$, for some $c\in\CH_2(X\times X)$, and its intersection with $\tilde I$ can be represented as $(q,q)^*(c\cdot \Delta_X)$. Since $c_1(Gr(3,10))\cdot{\Delta_X}$ is proportional to $\Delta_{Gr(3,10)}|_{X\times X}$, it is a cycle coming from $\CH(Gr(3,10)\times Gr(3,10))$. Therefore this part gives the term $\Gamma_2$ in the final relation.
\item the part proportional to $c_2^j(3)$. It will lead to the term
$$
(p,p)_*(q,q)^*\Gamma_3=(p,p)_*(q,q)^*(c_2(Gr(3,10)))
$$
in the final relation.
\end{enumerate}

\end{proof}

\subsection{Technical lemmas: open part}
We start with the study of the local geometry of $I_3$. Let $([W_1],[W_2],[W_3])$ be a general point of $I_3$. By definition of $I_3$, the three spaces
$$
K_1=W_2\cap W_3,\ K_2=W_3\cap W_1,\ K_3=W_1\cap W_2 
$$ 
are pairwise transversal. In particular, we have decompositions
$$
W_1=K_2\oplus K_3, W_2=K_3\oplus K_1, W_3=K_1\oplus K_2, V_9=K_1\oplus K_2\oplus K_3,
$$
and since $\alpha_X$ vanishes on $W_1, W_2, W_3$, the restriction $\alpha'=i^*_{V_9}\alpha_X$ belongs to $K_1^*\otimes K_2^*\otimes K_3^*$. We note that $\alpha'$ defines a hypersurface $X'$ in $Gr(3,V_9)$ and $X'$ contains $Z_{W_1}, Z_{W_2},$ and $Z_{W_3}$, where the notation $Z_{W_i}$  was introduced above. 

Let $O$ be an open chart of $Gr(3, K_1\oplus K_2\oplus K_3)$ defined in the following way:
$$
O=\left\{V:\dim \pi_{K_1}(V)=3\right\}.
$$
We note that $O$ is naturally isomorphic to the affine space $\mathrm {Hom}(K_1, W_1)$ or to $\mathrm {Hom}(K_1, K_2)\oplus \mathrm {Hom}(K_1, K_3)$. The following lemma shows that 
$\alpha'|_O$ is a quadratic form and hence $X'|_O$ is a quadratic hypersurface in $O$.

\begin{lemma}
\label{l_48}
Let $\alpha'\in K_1^*\otimes K_2^*\otimes K_3^*$. Its restriction $\alpha'|_O$ defines a pairing between $\mathrm {Hom}(K_1, K_2)$ and $\mathrm {Hom}(K_1, K_3)$, which in some basis can be represented by a $9\times 9-$matrix 
$$
\left(
\begin{matrix}
0&-Q_3&-Q_2\\
Q_3&0&-Q_1\\
Q_2&Q_1&0
\end{matrix}
\right),
$$
where $Q_1,Q_2,Q_3$ are $3\times 3$-matrices. Moreover, this pairing is non-degenerate for a general choice of $\alpha'$.
\end{lemma}

\begin{proof}
Let $(e_1,e_2,e_3)$ be a basis of $K_1$, $(e_4,e_5,e_6)$ be a basis of $K_2$, and $(e_7,e_8,e_9)$ be a basis of $K_3$. Let a point $p\in \mathrm {Hom}(K_1, W_1)$ be given by the matrix (\ref{matrix_V}). We evaluate $\alpha'$ on  the trivector 
\begin{multline*}
\left(e_1+(n_1e_4+n_2e_5+n_3e_6)+(m_1e_7+m_2e_8+m_3e_9)\right)\\
\wedge
\left(e_2+(n_4e_4+n_5e_5+n_6e_6)+(m_4e_7+m_5e_8+m_6e_9)\right)\\
\wedge
\left(e_3+(n_7e_4+n_8e_5+n_9e_6)+(m_7e_7+m_8e_8+m_9e_9)\right).
\end{multline*}

As $\alpha'\in K_1^*\otimes K_2^*\otimes K_3^*$, we can as well evaluate $\alpha'$ on
\begin{multline*}
e_1\wedge (n_4e_4+n_5e_5+n_6e_6)\wedge(m_7e_7+m_8e_8+m_9e_9)\\
-e_1\wedge(n_7e_4+n_8e_5+n_9e_6)\wedge(m_4e_7+m_5e_8+m_6e_9)\\
+e_2\wedge(n_7e_4+n_8e_5+n_9e_6)\wedge(m_1e_7+m_2e_8+m_3e_9)\\
-e_2\wedge(n_1e_4+n_2e_5+n_3e_6)\wedge(m_7e_7+m_8e_8+m_9e_9)\\
+e_3\wedge(n_1e_4+n_2e_5+n_3e_6)\wedge(m_4e_7+m_5e_8+m_6e_9)\\
-e_3\wedge(n_4e_4+n_5e_5+n_6e_6)\wedge(m_1e_7+m_2e_8+m_3e_9).
\end{multline*}
On the other hand, $\alpha'$ can be written as 
$e_1^*\wedge Q_1+e_2^*\wedge Q_2+e_3^*\wedge  Q_3$, where $Q_i\in K_2^*\otimes K_3^*$. This gives the desired matrix presentation.

Assuming that $Q_3$ is non-degenerate and using operations on lines, we can transform the matrix to 
$$
\left(
\begin{matrix}
0&1&Q_3^{-1}Q_2\\
1&0&-Q_3^{-1}Q_1\\
0&0&-Q_2Q_3^{-1}Q_1+Q_1Q_3^{-1}Q_2
\end{matrix}
\right).
$$
Now we see that the pairing is non-generate if and only if $-Q_2Q_1^{-1}Q_3+Q_3Q_1^{-1}Q_2$ is non-degenerate. This condition is an open condition and it is true for the following choice:
$$
Q_1=
\left(
\begin{matrix}
1&0&0\\
0&2&0\\
0&0&3
\end{matrix}
\right),
Q_2=
\left(
\begin{matrix}
0&1&0\\
0&0&1\\
1&0&0
\end{matrix}
\right),
Q_3=
\left(
\begin{matrix}
1&0&0\\
0&1&0\\
0&0&1
\end{matrix}
\right).
$$
Hence it is true for a general choice of $Q_1,Q_2,Q_3$.
\end{proof}

\begin{proposition}
\label{triangle}
\begin{enumerate}
\item \label{triangle_small}
The 9-dimensional cycle $Z_{W_1}+Z_{W_2}+Z_{W_3}$  in $\CH_9(X')$ is the restriction of a cycle $Z$ of $Gr(3,K_1\oplus K_2\oplus K_3)$.
\item \label{triangle_big} The 9-dimensional cycle $Z_{W_1}+Z_{W_2}+Z_{W_3}$ in $\CH_9(X)$ is the restriction of a cycle $Z'$ of $Gr(3,10)$. In particular, $Z_{W_1}+Z_{W_2}+Z_{W_3}$ is constant in $\CH_9(X)$, i.e., it does not depend on the choice of 
$([W_1],[W_2],[W_3])\in I_3$.
\end{enumerate}
\end{proposition}

\begin{proof}
The statement (\ref{triangle_big}) is an immediate consequence of (\ref{triangle_small}). We prove (\ref{triangle_small}).

Let us represent $Gr(3, K_1\oplus K_2\oplus K_3)$ as the union of the chart $O$ (as above), and subvarieties $D_1$, $D_2$, $D_3$ where
$$
O=\left\{V:\dim \pi_{K_1}(V)=3\right\},
$$
$$
D_k=\left\{V:\dim \pi_{K_1}(V)=3-k\right\}.
$$

Due to Lemma \ref{l_48}, the form $\alpha'|_O$ defines a non-degenerate quadratic hypersurface $Q=X'|_O$. Using this fact we are going to represent cycle $(Z_{W_2}+Z_{W_3})|_O$ as the sum 
$$
(B_1-B_2+B_3-\ldots+B_9)|_{Q},
$$
where $B_i$ are $10-$dimensional subspaces of $O$. We can take
$$
B_i=\left<v_i,v_{i+1},\ldots,v_9,v^*_1,\ldots,v^*_{i} \right>,
$$
where $(v_1,v_2,\ldots,v_9)$ is an arbitrary basis of $K_2^3$ and $(v_1^*,v_2^*,\ldots,v_9^*)$ is the dual basis in $K_3^3$. 
Now if $A_1=K_2^3$, $A_{10}=K_3^3$, and
$$
A_i=\left<v_i,v_{i+1},\ldots,v_9,v^*_1,\ldots,v^*_{i-1} \right>
$$
for $2\leq i\leq 9$, it is easy to see that $A_i$ are $9-$dimensional affine spaces contained in hypersurface $Q$. We also have $A_1=Z_{W_2}|_O$ and $A_{10}=Z_{W_3}|_O$. Moreover, the restriction $B_i|_{Q}$ is $A_i\cup A_{i+1}$ for each $i$. Therefore $(B_1-B_2+B_3-\ldots+B_9)|_{Q}=(Z_{W_2}+Z_{W_3})|_O$.

Let $\overline{B}_i\subset Gr(3,V_9)$ be the Zariski closure of $B_i\subset O$. We are going to take the cycle $(\overline B_1-\overline B_2+\overline B_3-\ldots+\overline B_9)$ as the desired cycle in $Gr(3,V_9)$. To finish the proof we need to investigate the boundaries $B'_{ik}=\overline B_i\cap D_k$.
For $k=1,2$, we will prove in Lemma \ref{l_47} that the intersection of $B'_{ik}$ and $X'\cap D_k$ has dimension at most $8$ for a general choice of the basis $(v_1, \ldots,v_9)$. For $k=3$, we note that $Z_{W_2}$ identifies to $\mathrm {Hom}(K_1,K_3)$, via the isomorphism $W_2=K_1\oplus K_3$, therefore the cycle $Z_{W_1}$ is contained in the complement of $O$ and in fact equal to $D_3$.
We have an inclusion:
$$
\overline B_i\cap X'\subset\left(\overline A_i\cup \overline A_{i+1}\cup Z_{W_1}\right)\cup\left(\mathrm{lower\ dimensional\ terms}\right).
$$
Therefore the restriction to $X'$ of the closure of $B_1-B_2+B_3-\ldots+B_9$ defines a cycle of the form $dZ_{W_1}+Z_{W_2}+Z_{W_3}$ for some
$d\geq 0$.  Permuting the $W_i$ and adding up, we conclude that $(d+2)(Z_{W_1}+Z_{W_2} +Z_{W_3})$ is the restriction of a cycle of $Gr(3,V_9)$, which concludes the proof.
\end{proof}

In the end of this subsection, we are going to present a lemma about the relation of $I$ and $I_3$.

\begin{lemma}
\label{lemma_I_6}
For a general $X$ the expected dimension of $I$ is 6.
\end{lemma}

\begin{proof}
The codimension of $\Delta_X$ is 20, therefore the codimension of $(q,q)^{-1}\Delta_X$ is also 20. Since $U\times U$ has dimension 26, the dimension of $(q,q)^{-1}\Delta_X$ as well as the dimension of $(p,p)(q,q)^{-1}\Delta_X$ should be 6. 
\end{proof}

\begin{lemma}
\label{l_45}
\label{l_45_2}
For a general choice of $\alpha_X$, the natural projection $\pi_{12}: I^o_3\to I$ has degree one. 
As a consequence, there exists a birational map:
$$
\xymatrix{
\phi=\pi_3\circ \pi_{12}^{-1}:I\ar@{-->}[r] & F(X).
}$$ 
\end{lemma}


\begin{proof}
We are going to understand the fiber $\pi_{12}^{-1}(p)$ for a general point $p=([W_1],[W_2])\in I$. Let $K_3 = W_1\cap W_2$, $K_2\subset W_1$ such that $K_2\cap K_3 = 0$, and $K_1\subset W_2$ such that $K_1\cap K_2=0$, $V_9=K_1\oplus K_2\oplus K_3$.  

Any point of $I^o_3\cap \pi_{12}^{-1}(p)$ belongs to the following open char of $Gr(6,V_9)$:
$$
\{V : \dim\pi_{K_1\oplus K_2}(V)=6\} \subset Gr(6,V_9).
$$
This chart can be identify with $\mathrm{Hom}\, (K_1\oplus K_2, K_3)$, and its
point $(\phi\oplus\psi)$ belongs to $I_3^o$, if and only if the following equations hold:
$$
\alpha' ( (v_1+\phi(v_1) )\wedge(v_2+\phi(v_2))\wedge (w_1+\psi(w_1)) ) = 0,
$$
$$
\alpha' ( (v_1+\phi(v_1) )\wedge(w_1+\psi(w_1))\wedge (w_2+\psi(w_2)) ) = 0
$$
for any $v_1,v_2\in K_1$ and any $w_1,w_2\in K_2$. Since, $i^*_{W_1}\alpha'=i^*_{W_2}\alpha'=0$, these equations are equivalent to
$$
\alpha'( v_1\wedge \phi(v_2)\wedge w_1) + \alpha'(\phi(v_1) \wedge v_2\wedge w_1)=-\alpha'( v_1\wedge v_2\wedge w_1).
$$
$$
\alpha'( v_1\wedge \psi(w_1)\wedge w_2) + \alpha'(v_1 \wedge w_1\wedge \psi(w_2))=-\alpha'( v_1\wedge w_1\wedge w_2).
$$
The equations define two linear systems: for $\phi$ and for $\psi$. Since the coefficients in the left hand sides of equations are defined only by the part of $\alpha'$ which belongs to $K_1^*\otimes K_2^*\otimes K_3^*$, we may apply Lemma \ref{l_48} to see that the matrix of coefficients in the linear system for $\phi\in\mathrm{Hom}\, (K_1,K_3)$ is a non-degenerate $9\times9-$matrix. Similarly, the matrix for $\psi$ is non-degenerate. Therefore there is a unique solution and the natural projection $\pi_{12}: I^o_3\to I$ has degree one.
\end{proof}

\subsection{Technical lemmas: boundary}

The goal of this subsection is to prove the following lemma. We continue with the notation from the previous subsection.

\begin{lemma}
\label{l_47}
For a general choice of basis $(v_j)$ of $K_2^3$ and for any choice of integer $i$ between $1$ and $9$, let $B\subset O$ be the $10$-dimensional vector space
$$
\left<v_i,v_{i+1},\ldots,v_9,v^*_1,\ldots,v^*_{i} \right>,
$$
where $v_j^*$ denotes the dual basis of $K_3^*$ with respect to $\alpha'|_{O}$. Then the intersection $\overline B\cap D_k\cap X'$ has dimension at most $8$ for $k=1,2$.
\end{lemma}

The proof of Lemma \ref{l_47} will rest on Lemma \ref{l_49} for $k=1$ and Lemma \ref{l_411} for $k=2$. Before the proof we introduce local coordinates on $O$ and relate them to the local coordinates on $D_1$ and on $D_2$. In particular, we show that any point on $D_k$ is the limit point of some affine line in $O$. Finally, we show that the intersection $\overline B\cap D_k\cap X'$ has dimension at most $8$ by proving that any $9-$dimensional component of $\overline B\cap D_k$ is not contained in $X'$.

We recall that a point in $Gr(3,9)$ can be represented by three independent vectors, i.e., $3\times 9$ matrix (both up to $GL(3)-$action). Fixing basis of $K_1, K_2$, and $K_3$, we have in chart $O$ a representation 
\begin{equation}
\label{matrix_V}
p_O=\left(
\begin{matrix}
1&0&0&n_1&n_2&n_3&m_1&m_2&m_3\\
0&1&0&n_4&n_5&n_6&m_4&m_5&m_6\\
0&0&1&n_7&n_8&n_9&m_7&m_8&m_9\\
\end{matrix}
\right).
\end{equation}
In this notation, $n_i$ correspond to an element of $\mathrm{Hom}(K_1, K_2)$ and $m_i$ correspond to an element of $\mathrm{Hom}(K_1,K_3)$. Unfortunately, it is not possible to relate coordinates $n_j$ and $m_j$ with the basis $(v_j,v_j^*)$ in a simple way, because the quadratic form $\alpha'|_O$ is not a general quadratic form on $O$ (its form was explained in Lemma \ref{l_48}).

Now we are going to study the boundaries $D_1$ and $D_2$.  According to the definition of $D_k$, a point $p\in D_k$ can be represented by a $3\times 9-$matrix, whose rank of the first three columns is equal to $3-k$. 

For $k=1$, we need another chart $O'$ where the rank of the first three columns may be two. Without loss of generality, we may assume that the columns number $1, 2$, and $4$ are linearly independent in $O'$. On $O'\cap O$ we have 
\begin{multline*}
\left(
\begin{matrix}
1&0&n'_1&0&n'_2&n'_3&m'_1&m'_2&m'_3\\
0&1&n'_4&0&n'_5&n'_6&m'_4&m'_5&m'_6\\
0&0&n'_7&1&n'_8&n'_9&m'_7&m'_8&m'_9\\
\end{matrix}
\right)=\\
\left(
\begin{matrix}
1&0&-\frac{n_1}{n_7}&0&\frac{n_2n_7-n_1n_8}{n_7}&\frac{n_3n_7-n_1n_9}{n_7}&\frac{m_1n_7-n_1m_7}{n_7}&\frac{m_2n_7-n_1m_8}{n_7}&\frac{m_3n_7-n_1m_8}{n_7}\\
0&1&-\frac{n_4}{n_7}&0&\frac{n_5n_7-n_4n_8}{n_7}&\frac{n_6n_7-n_4n_9}{n_7}&\frac{m_4n_7-n_4m_7}{n_7}&\frac{m_5n_7-n_4m_8}{n_7}&\frac{m_6n_7-n_4m_8}{n_7}\\
0&0&\frac{1}{n_7}&1&\frac{n_8}{n_7}&\frac{n_9}{n_7}&\frac{m_7}{n_7}&\frac{m_8}{n_7}&\frac{m_9}{n_7}\\
\end{matrix}
\right).
\end{multline*}
The intersection $D_1'=O'\cap D_1$ is the $17-$dimensional affine space defined by $n'_7=0$ in $O'$. Any point $p\in D_1'$ can be represented as the $(t\rightarrow\infty)$-limit of the affine line 
$$
l=t\left(0|N_0|M_0\right)+\left(id|N_1|M_1\right)
$$
in $O$, where the $3\times 3-$matrices $N_0, N_1, M_0, M_1$ are defined by the coordinates $n'_i$ and $m'_i$ of $p$. Moreover, since the coordinates $n'_2,n'_3,n'_5,n'_6$ and $m'_1,\ldots,m'_6$ are well-defined as $t\to \infty$, the direction $(N_0|M_0)$ must satisfy the condition $rk(N_0|M_0)=1$. Conversely, a choice of $(N_0|M_0)$ with $rk(N_0|M_0)=1$ and a point $(N_1|M_1)\in O$ defines a line, whose limit point $p$ belongs to $D_1$. The coordinates
$$
n_1', n'_4,n'_7, n'_8,n'_9, m'_7, m'_8, m'_9
$$
of the limit point $p\in D_1'$ are defined only by the choice of $(N_0|M_0)$ and the remaining coordinates 
$$
n'_2, n'_3, n'_5, n'_6, m'_1,\ldots, m'_6
$$
are defined by $(N_1|M_1)$, while $(N_0|M_0)$ is fixed. The projection of $D'_1$ along the coordinates $n'_2, n'_3, n'_5, n'_6, m'_1,\ldots, m'_6$ defines the structure of a fibration on $D_1'$, which can be extended to $D_1$.

\begin{lemma}
\label{l_49}
Let $F$ be the intersection of $\overline B$ with a fiber of $D_1$. Then there exists a point $p\in F$, such that $p\notin X'$.
\end{lemma}
\begin{proof} 
A choice of fiber is equivalent to the choice of $(N_0|M_0)$ such that $rk(N_0|M_0)=1$. We note that $\alpha'(id|N_0|M_0)=0$, therefore $(N_0|M_0)$ belongs to $B\cap X'$, which is the union of two subspaces 
$$
\left<v_i,v_{i+1},\ldots, v_9,v_1^*,\ldots v_{i-1}^*\right>\ \mathrm{and}\ \left<v_{i+1},\ldots, v_9,v_1^*,\ldots v_{i}^*\right>.
$$
Without loss of generality, we may assume that $(N_0|M_0)$ belongs to the first subspace. In this case $v_i^*$ is not related with $(N_0|M_0)$ and can be chosen arbitrarily (more precisely, there is a choice of basis $(v_j)$ providing the given choice of $v_i^*$). It allows to take a point $(N_1|M_1)$ in $B$ corresponding to $v_i^*$. In this case $N_1=0$ and $M_1$ is arbitrary, which makes the coordinates $m'_1,\ldots,m_6'$ arbitrary (while $n_1', n'_4,n'_7, n'_8,n'_9, m'_7, m'_8, m'_9$ are fixed by the choice of $(N_0|M_0)$). 
Representing $\alpha'$ as $Q'_1\wedge e_7^*+Q'_2\wedge e_8^*+Q'_3\wedge e_9^*$, where $Q_i'\in K_1^*\otimes K_2^*$, we see that
\begin{multline*}
\alpha'(p)=m_1 Q'_1((0,1,n'_4), (1,n'_8,n'_9))+m_2 Q'_2((0,1,n'_4), (1,n'_8,n'_9))+\\
m_3 Q'_3((0,1,n'_4), (1,n'_8,n'_9))+\mathrm {other}\ \mathrm{terms}. 
\end{multline*}
Since $(0,1,n'_4)$ can not be orthogonal to $(1,n'_8,n'_9)$ with respect to all three forms $Q_i'$, at least one of the first three terms is not zero. Therefore a general choice of $m'_1,\ldots,m_6'$ provides a point $p$ with $\alpha(p)\neq0$ and thus $p\not\in X'$.
\end{proof}

For $k=2$, we need a chart $O'$, where the rank of the first three columns may be one. Without loss of generality, we may assume that the columns number $1,4$, and $5$ are linearly independent in $O'$. On $O'$ we have coordinates
$$
\left(
\begin{matrix}
1&n'_1&n'_2&0&0&n'_3&m'_1&m'_2&m'_3\\
0&n'_4&n'_5&1&0&n'_6&m'_4&m'_5&m'_6\\
0&n'_7&n'_8&0&1&n'_9&m'_7&m'_8&m'_9\\
\end{matrix}
\right).
$$
The intersection $D_2'=D_2\cap O'$ is the $14-$dimensional affine subspace in $O'$, defined by $n'_4=n'_5=n'_7=n'_8=0$. Again, any point on $D_2'$ can be represented as the limit of the affine line 
$$
l=t\left(0|N_0|M_0\right)+\left(id|N_1|M_1\right)
$$
in $O$ and, similarly to the case $k=1$, we have $rk(N_0|M_0)=2$. Conversely, an affine line $t\left(0|N_0|M_0\right)+\left(id|N_1|M_1\right)$ with $rk(N_0|M_0)=2$ defines a limit point $p\in D_2$. We also have the structure of a fibration on $D_2$: we can fix $(N_0|M_0)$ and then vary $(N_1|M_1)$. A fiber has the following matrix presentation in $O'$:
$$
\left(
\begin{matrix}
1&n'_1&n'_2&0&0&*&*&*&*\\
0&0&0&1&0&n'_6&m'_4&m'_5&m'_6\\
0&0&0&0&1&n'_9&m'_7&m'_8&m'_9\\
\end{matrix}
\right),
$$
where $n_j'$ and $m_j'$ are fixed. We note that such a fiber of $D_2$ is entirely contained in $X'$ or has an empty intersection with $X'$. It makes sense to consider a projection
$$
\pi: O'\to P,
$$
along the coordinates $n_3', m'_1, m'_2, m'_3$, where $P\subset O'$ is the linear subspace defined by $n_3'=m'_1=m'_2=m'_3=0$. Then, any point $p(t)$ on the line 
$$
t\left(0|N_0|M_0\right)+\left(id|0|0\right),
$$
passing through the origin of $O$ with $rk(N_0|M_0)=2$, can be represented as
$$
\left(
\begin{matrix}
1&n_1'&n_2'&0&0&0&0&0&0\\
0&1&0&tn_4&tn_5&tn_6&tm_4&tm_5&tm_6\\
0&0&1&tn_7&tn_8&tn_9&tm_7&tm_8&tm_9\\
\end{matrix}
\right)
$$
or, for $t\neq 0$,
$$
\left(
\begin{matrix}
1&n_1'&n_2'&0&0&0&0&0&0\\
0&1/t&0&n_4&n_5&n_6&m_4&m_5&m_6\\
0&0&1/t&n_7&n_8&n_9&m_7&m_8&m_9\\
\end{matrix}
\right).
$$
In this notation the value $\alpha'(p(t))$ does not depend on $t$. Therefore the limit point of such a line belongs to $X'\cap D_2$ if and only if the line is contained in $X'$.

We are going to show that the intersection $\overline B\cap D_2$ is not contained in $X'$. We need the following auxiliary lemma.

\begin{lemma}
\label{l_410}
Let $1\leq i\leq 9$ and let
$$
A=\left<v_i,\ldots,v_9,v_1^*,\ldots,v_{i-1}^*\right>.
$$
Then $\pi(\overline A\cap D'_2)$ has dimension at most four.
\end{lemma}

\begin{proof}
The dimension of $\pi(\overline A\cap D'_2)$ depends only on the variety of "directions" in $A$:
$$
\mathcal M=\{(N_0|M_0)\in A: rk(N_0|M_0)=2\}.
$$
Let $\mathcal M_0$ be an irreducible component of $\mathcal M$. There are two possibilities:
\begin{enumerate}
\item $M_0=0$ for general $(N_0|M_0)\in \mathcal M_0$. In this case we have $m'_1=m_2'=\ldots=m'_6=0$ for the corresponding component of $\pi(\overline A\cap D'_2)$. Hence its dimension is at most $4$.

\item $M_0\neq0$ for general $(N_0|M_0)\in \mathcal M_0$. Let $\mathcal M_{0N}$ be the image of the projection of $\mathcal M_0$ to $K_2^3$, i.e., $\mathcal M_{0N}$ is the variety of $N_0$ in pairs $(N_0|M_0)\in \mathcal M_{0}$. We note that $\mathcal M_{0N}$ is contained in the cubic hypersurface defined in $\left<v_i,\ldots, v_9\right>$ by $det(N_0)=0$. In particular, $\dim \mathcal M_{0N}\leq 9-i$. Similarly, we can define $\mathcal M_{0M}$ as the image of the projection of $\mathcal M$ to $K_3^3$. We also have $\dim \mathcal M_{0M}\leq i-2$. There are two possibilities:
\begin{enumerate}
\item $\dim \mathcal M_{0N}=9-i$ or $\dim \mathcal M_{0M}=i-2$. Without loss of generality, we may assume that $\dim \mathcal M_{0N}=9-i$. We are going to show that the fiber of $\mathcal M_0\to \mathcal M_{0N}$ over a general point of $\mathcal M_{0N}$ has dimension at most $i-4$. This will imply that $\dim \mathcal M_0\leq 5$ and therefore $\dim \pi(\overline A\cap D_2)\leq 4$.

Let $N_0$ be a general point of $\mathcal M_{0N}\subset \left<v_i,\ldots, v_9\right>$. The fiber of $\mathcal M_0$ over $N_0$ is the subvariety of  
$$
A\cap K_3^3=\left<v_1^*,\ldots,v_{i}^*\right>
$$
consisting of all matrices $M_0$ such that $rk(N_0|M_0)=2$. It can be defined by $3$ linear equations (depending on $N_0$) and hence it can be represented as the subspace orthogonal with respect to the quadratic form $\alpha'|_O$ to some $3-$dimensional subspace
$$
\left<f_1, f_2, f_3\right>\subset K_2^3=\left<v_1,\ldots,v_{9}\right>,
$$
where $f_1,f_2,f_3$ depend on $N_0$. It can be also seen as the subspace of 
$$
K_3^3=\left<v_1^*,\ldots,v_{9}^*\right>
$$
orthogonal to 
$$
\left<f_1, f_2, f_3, v_i, v_{i+1},\ldots,v_9\right>.
$$
For a general choice of $\alpha$ and a general choice of $N_0$, the vectors 
$$
N_0, f_1,f_2,f_3
$$
are linearly independent in $K_2^3$. Moreover, we claim that all $13-i$ vectors 
$$
f_1,f_2,f_3, v_i,\ldots, v_9
$$
are linearly independent. Indeed, this condition is an open condition, and it is enough to show the result for some choice of $v_j$. We can fix $N_0$ in $K_3^3$, take $v_i=N_0$, then define $v_j$ for $j\neq i$ in such a way that the desired $13-i$ vectors are linearly independent. It is possible, while the dimension of $\left<f_1,f_2,f_3, v_i,\ldots, v_9\right>$ is less than $9$. The assumption of the existence of non-zero $M_0$ in the fiber, provides that $\dim\left<f_1,f_2,f_3, v_i,\ldots, v_9\right>\leq8$.

The independence of $f_1,f_2,f_3, v_i,\ldots, v_9$ will imply that the dimension of a fiber over $N_0$ is at most $i-4$. Hence $\dim \mathcal M_0\leq 5$ and therefore $\dim \pi(\overline A\cap D'_2)\leq 4$.


\item $\dim \mathcal M_{0N}\leq 8-i$ and $\dim \mathcal M_{0M}\leq i-3$. In this case, $\dim \mathcal M_0\leq 5$ and therefore $\dim \pi(\overline A\cap D'_2)\leq 4$.
\end{enumerate}
\end{enumerate}

\end{proof}

\begin{lemma}
\label{l_411}
No $9-$dimensional component of $\overline B\cap D_2$ is contained in $X'$.
\end{lemma}

\begin{proof}
We prove the lemma by contradiction. Assume that there is a component $B'$ of $\overline B\cap D_2$ contained in $X'$. Hence $\pi(B')$ is contained in $X'$ and therefore each point of $\pi(B')$ is the limit of a line contained in $B\cap X'$. Since $B\cap X'=A\cup A'$, where
$$
A=\left<v_i,\ldots v_9,v_1^*,\ldots, v_{i-1}^{*}\right>, A'=\left<v_{i+1},\ldots v_9,v_1^*,\ldots, v_{i}^{*}\right>.
$$
We have that $\pi(B')$ is contained in $\pi(\overline A\cup \overline A')$. Due to Lemma \ref{l_410}, $\dim\pi(\overline A\cup \overline A')\leq 4$, but $\dim\pi(B')\geq 5$. We have a contradiction, which proves the lemma.
 \end{proof}

\begin{proof}[Proof of Lemma \ref{l_47}]
For $k=1$ the result follows from Lemma \ref{l_49}. The intersection $F$ of $\overline B$ with a fiber of $D_1$ is irreducible and therefore it is contained in one irreducible component of $\overline B\cap D_1$. Conversely, any component contains such a fiber. Since each fiber contains a point which does not belong to $X'$, each irreducible component of $\overline B\cap D_1$ is not contained in $X'$.   

For $k=2$ the result follows from Lemma \ref{l_411}. 
\end{proof}



\bibliographystyle{abbrv}

\begin{thebibliography}{10}

\bibitem{B84}
A.~Beauville.
\newblock Vari\'et\'es k\"ahl\'eriennes dont la premiere class de chern est
  nulle.
\newblock {\em J. Differential Geom.}, 18:755--782, 1984.

\bibitem{BD}
A.~Beauville and R.~Donagi.
\newblock La vari\'et\'e des droits d'une hypersurface qubique de dimension 4.
\newblock {\em C. R. Acad. Sci. Paris S\'er. I Math}, 301(14):703--706, 1985.

\bibitem{DV}
O.~Debarre and C.~Voisin.
\newblock {Hyper-Kaehler fourfolds and grassmann geometry}.
\newblock {\em Journal f\"ur die reine und angewandte Mathematik}, 649, 2010.

\bibitem{H14}
D.~Huybrechts.
\newblock {Curves and cycles on K3 surface}.
\newblock {\em Algebraic Geometry}, 1:69--106, 2014.

\bibitem{IR1}
A.~Iliev and K.~Ranestad.
\newblock K3 surface of genus 8 and varieties of sums of powers of cubic
  fourfolds.
\newblock {\em Trans. Amer. Math. Soc.}, 353(4):1455--1458, 2001.

\bibitem{IR2}
A.~Iliev and K.~Ranestad.
\newblock {Addendum to "K3 surface of genus 8 and varieties of sums of powers
  of cubic fourfolds"}.
\newblock {\em C.R. Acad. Bulgare Sci.}, 60(12):1265--1270, 2007.

\bibitem{Lin15}
H.-Y. Lin.
\newblock {On the Chow group of zero-cycles of a generalized Kummer variety}.
\newblock arXiv:1507.05155v1, 2015.

\bibitem{M}
D.~Mumford.
\newblock Rational equivalence of zero-cycles on surfaces.
\newblock {\em J. Math. Kyoto Univ.}, (9):195--204, 1968.

\bibitem{OG2}
K.~O'Grady.
\newblock {Desingularized moduli spaces of sheaves on a K3}.
\newblock {\em J. reine angew. Math}, 512:49--117, 1999.

\bibitem{OG1}
K.~O'Grady.
\newblock A new six-dimensional irreducible symplectic variety.
\newblock {\em J. Alg. Geom}, 12:435--505, 2003.

\bibitem{OG4}
K.~O'Grady.
\newblock {Irreducible symplectic 4-folds and Eisenbud--Popescu--Walter
  sextics}.
\newblock {\em Duke Math. J.}, 134(1):99--137, 2006.

\bibitem{SV}
M.~Shen and C.~Vial.
\newblock {The Fourier Transform for Certain Hyper-Kaehler Fourfolds}.
\newblock {\em Memoirs of the AMS}, 240(1139), 2016.

\bibitem{Voisin08}
C.~Voisin.
\newblock {On the Chow ring of certain algebraic hyper-K\"ahler manifolds}.
\newblock {\em Pure Appl. Math. Q.}, 4(3):613--649, 2008.

\bibitem{Voisin15}
C.~Voisin.
\newblock {Remarks and questions on coisotropic subvarieties and zero-cycles of
  hyper-K\"ahler varieties}.
\newblock arXiv:1501.02984, 2015.

\end{thebibliography}

\end{document}